\documentclass[12pt]{article}

\usepackage[english]{babel}
\usepackage{authblk}
\usepackage[utf8]{inputenc}
\usepackage[T1]{fontenc}

\usepackage{enumitem}
\usepackage{nccmath}
\usepackage{amsthm} 
\usepackage{amsmath}
\usepackage{mathrsfs}
\usepackage{amssymb}
\usepackage{mathabx}

\usepackage{fancyhdr}
\usepackage{epsfig}
\usepackage{epic}
\usepackage{eepic}
\usepackage{threeparttable}
\usepackage{amscd}
\usepackage{here}
\usepackage{graphicx}
\usepackage{lscape}
\usepackage{tabularx}
\usepackage{subfigure}
\usepackage{longtable}
\usepackage{makeidx}
\usepackage{tikz-cd}
\usepackage[T1]{fontenc}
\usepackage{textcomp}
\usepackage{thmtools}
\usepackage{hyperref}
\usepackage{babelbib}
\usepackage{setspace}
\usepackage{marginnote}
\usepackage{multicol}
\usepackage{authblk}

\theoremstyle{definition}
\newtheorem{Teo}{Theorem}[section]
\newtheorem{Prop}[Teo]{Proposition}
\newtheorem{Coro}[Teo]{Corollary}

\newtheorem{Obs}[Teo]{Remark}
\newtheorem{Lema}[Teo]{Lemma}
\newtheorem{Def}[Teo]{Definition}
\newtheorem{Ej}[Teo]{Example}

\date{}

\begin{document}
\title{Constructible sets in lattice-valued models: A negative result.}

\author{Jose Moncayo\thanks{jrmoncayov@unal.edu.co} }
\author{Pedro H. Zambrano\thanks{phzambranor@unal.edu.co}}
\affil{Departamento de Matem\'aticas, Universidad Nacional de Colombia, AK 30 $\#$ 45-03 c\'odigo postal 111321, Bogota, Colombia.}
\date{\today}
\maketitle

\begin{abstract}
    We investigate different set-theoretic constructions in Residuated Logic based on Fitting's work on Intuitionistic Set Theory (\cite{Fitting1969}).

    We start by stating some results concerning constructible sets within valued models of Set Theory. We present two distinct constructions of the constructible universe: $\mathfrak{L}^{\mathbb{Q}}$ and $ \mathbb{L}^{\mathbb{Q}}$, and show that they are isomorphic to $V$ (the classical von Neumann universe) and $L$ (the classical Gödel's constructible universe), respectively.
    
    Even though lattice-valued models are the natural way to study non-classical Set Theory (e.g., Intuitionistic, Residuated, Paraconsistent Set Theory, see \cite{Grayson1979, Lano1992a, VenturiMartinez2021,CarnielliConiglio2019}), our results prove that the use of lattice-valued models is not suitable to study the notion of constructibility in logics weaker than classical logic.

\end{abstract}

\section{Introduction}

The notion of constructibility (in Set Theory) started with Gödel's work \cite{Godel1938} on the consistency of the Axiom of Choice ($AC$) and the Generalized Continuum Hypothesis ($GCH$). Gödel considers the class of \textbf{definable} sets in \textbf{Classical first-order Logic} in the language of Set Theory, now called Gödel's constructible universe. 

At the beginning, Gödel's idea of considering definable sets in some logic was not widely used in the construction of new inner models, but rather, using different set-theoretical techniques, new inner models such as $HOD$ or $L[\mathcal{U}]$ were defined that allowed the advancement of Set Theory, especially in the realm of independence results. Nonetheless, a couple of attempts were made to generalize Gödel's idea of a class of definable sets: Scott and Myhill \cite{Scott1971} showed that the well-known model $HOD$ can be obtained by using the definable sets in \textbf{second-order logic} and Chang \cite{Chang1971} showed that if one considers the definable sets in the \textbf{infinitary logic} $L_{\omega_1,\omega_1} $, an inner model is obtained that is characterized by being the smallest inner model that is closed under countable sequences. Although these results are interesting, no meaningful study of inner models arising from different logics was considered for a very long time.

It was not until the work of Kennedy, Magidor and Väänänen \cite{KMV2016} that inner models of Set Theory that arise when considering definable sets in generalized logics were systematically studied. The logics considered were \textbf{strengthenings} of first-order logic constructed by using generalized quantifiers or by allowing infinite disjunctions, conjunctions, or quantification. Some notable examples include the Stationary Set Theory, logics with cofinality quantifiers, the Härtig quantifier and the Magidor-Malitz quantifier. Such models made it possible to study new independence results in Set Theory.

Therefore, one could ask if such constructions could be done in logics that are \textbf{weakenings} (rather than strengthenings, as we just saw) of Classical first-order Logic. We would like to study logics general enough to capture the most important logical examples, such as \textbf{Intuitionistic} and \textbf{Fuzzy Logic}, but not so general that we lose too many structural rules, such as the commutativity of the premises in a deduction. Therefore, we are interested in studying constructibility in the context of the so called \textbf{substructural logics} without contraction (but with the exchange rule). Essentially, we consider a weakening of Intuitionistic Logic in which we consider two types of conjunctions: $\&$ and $\land$. The strong conjunction (denoted $\&$) is no longer idempotent, that is, 

\begin{center}
    $\alpha\rightarrow(\alpha\&\alpha)$
\end{center}
no longer holds for all formulas $\alpha$. The defining feature of this connective is that is the left adjoint to the implication (just as $\land$ is for the classical and intuitionistic case),
\begin{center}
    $\alpha\&\beta\Rightarrow \gamma$ if and only if $\alpha\Rightarrow \beta\rightarrow \gamma$.
\end{center}

We also consider a weak conjunction (denoted $\land$) closer to the intuitionistic one, but that is not necessarily the adjoint to the implication. 

Therefore, in this logic, called \textbf{Residuated Logic} (also \textbf{Monoidal Logic} in \cite{Hohle1994}), it is the case that the strength of the premises changes depending on how many of the same hypothesis we have (due to the lack of idempotency), such as it can be seen in the Deduction Theorem for Residuated Propositional Logic ($RPC$):

\begin{Teo}[Deduction theorem for $RPC$, \cite{MacCaull1996}]
    If there is a deduction in $RPC$ of $\theta$ from the set of formulas $\alpha, \beta, ..., \gamma, \delta$, and the deduction used $\delta$ $n$-times, then there is a deduction in $RPC$ of $(\delta^n\rightarrow \theta)$ from $\alpha, \beta, ..., \gamma$ (where $\delta^n=\delta \& \delta \& ... \& \delta$ is the $n$-fold conjunction of $\delta$ with itself).
\end{Teo}

Also in regards to the equality, we have that the usual substitution of equal elements

\begin{center}
    $(x=y)\& \theta(x)\Rightarrow \theta(y)$
\end{center}

is not going to hold in general, but rather, we have that

\begin{center}
    $(x=y)^n\& \theta(x)\Rightarrow \theta(y)$
\end{center}

where $n$ occurrences of $x$ in $\theta(x)$ that have been replaced by $y$ to form $\theta(y)$.

With these key features in mind, one could ask what kind of models are we going to use to study these logics, and more specifically, how do we can find natural models of Set Theory for these logics. A natural way to do this is considering \textbf{lattice-valued models}.

Lattice-valued models were first introduced by Scott and Solovay in \cite{SS1967}. They considered \textbf{Boolean-valued models of Set Theory} in order to provide a more intuitive presentation of \textbf{Cohen's forcing}. To achieve this, they took a complete Boolean algebra $\mathbb{B}$ and built a ``model'' of Set Theory $ V^{\mathbb{B}}$ in which the truth values of formulas take values in $\mathbb{B}$ instead of the trivial Boolean algebra $\{0,1\}$. 

Based on the construction of Scott and Solovay, several generalizations of the previous construction have been considered by taking other lattices instead of Boolean algebras. For example, Heyting lattices give rise to \textbf{Intuitionist models of Set Theory} \cite {Grayson1979}, $BL\Delta$-algebras give rise to \textbf{models of Fuzzy Set Theory} in the Fuzzy Logic $BL\forall\Delta$ \cite{HajekHanikova2001, HajekHanikova2003}, topological complete residuated lattices (i.e. topological commutative integral Quantales) give rise to \textbf{Modal models of Residuated Set Theory} \cite {Lano1992a} and . These kind of valued models serve as natural models of logics weaker than first-order.

Until now, as far as the authors are aware, there has been no in-depth study of what would be a ``class of definable sets'' in the context of valued models. The closest attempt to this was done by Fitting \cite{Fitting1969}, where it was shown how to construct $L$ (or more precisely a model ``isomorphic'' to $L$) using two-valued characteristic functions that are definable by some formula. This construction was introduced as a motivation for his definition of the class of constructible sets using Kripke models.

Following Fitting's idea, we propose new definitions of the notion of \textbf{definable subset} within a Boolean-valued model of Set Theory and with them, we propose two new constructions of the constructible universe: $\mathfrak{L}^{\mathbb{Q}}$ and $ \mathbb{L}^{\mathbb{Q}}$. Moreover, we prove that these models are in fact \textbf{two-valued}, since our definition of definability is too restrictive and forces the models to only take these values. Furthermore, we prove that $\mathfrak{L}^{\mathbb{Q}}$ and $ \mathbb{L}^{\mathbb{Q}}$ are ``isomorphic'' to $V$ (the classical von Neumann universe) and $L$ (the classical Gödel's constructible universe), respectively. 

When we tried to generalize these notions of definability to the context of Quantale-valued models, we found that the resulting classes of constructible sets are also \textbf{two valued} (citar), and therefore are not suitable to study constructibility within Residuated Logic. 

We start by introducing the algebraic structures 
notion of \textit{Quantale}

In Section \ref{completeresiduatedlattice}, we provide some basics on commutative integral Quantales, which is the underlying lattice behind the Residuated Logic.

In Section \ref{residuatedlogic}, we mention some generalities about Residuated Logic.

In Section \ref{valuedmodelsofsetheory}, we introduce the notion of Quantale-valued models of Set Theory in a similar way as in \cite{SS1967,Grayson1979,Lano1992a,VenturiMartinez2021} and discuss their relationship with several logics (Classical, Intuitionistic and Residuated Logic). In Subsection \ref{constructiblesetsonvaluedmodels}, we propose two definitions for the class of constructible sets in the context of Quantale-valued models of Set Theory, $\mathfrak{L}^{\mathbb{Q}}$ and $ \mathbb{L}^{\mathbb{Q}}$. We show that both $\mathbb{Q}$-valued models are in fact two-valued. Furthermore, we prove that $\mathfrak{L}^{\mathbb{Q}}$ and $ \mathbb{L}^{\mathbb{Q}}$ are ``isomorphic'' to the classical $V$ and $L$, respectively.

These results lead us to consider other kind of models (specifically, Kripke-like models) to study constructability in this approach (see~\cite{Moncayo2023}), which is done in the forthcoming paper \cite{MoncayoZambrano2X}.

\section{Commutative integral Quantales}\label{completeresiduatedlattice}

Structures like Quantales (i.e. ordered monoids with a product that distributes over arbitrary supremums) have been studied at least since Ward and Dilworth's work on residuated lattices \cite{WardDilworth1938, Dilworth1939, Ward1938}, where their motivations were more algebraic, since they were studying the lattice of ideals in a ring: Given a ring $R$, the set of ideals of $R$, denoted as $Id(R)$, forms a complete lattice defining infimum and supremum as the intersection and sum of ideals, respectively. The monoid operation $\cdot$ on this lattice is given by multiplication of ideals, and the element $R$ in $Id(R)$ is the identity of this operation.

 But it was not until the work of Mulvey \cite{Mulvey1986}, where the term \textbf{Quantale} was coined as a combination of \textbf{``quantum''} and \textbf{``locale''} and proposed their use for studying \textbf{Quantum Logic} and non-commutative $C^*$-algebras.

 Our motivation for the study of Quantales is somewhat different. We are not interested in Quantales that are non-commutative - as was the case for Mulvey - but rather Quantales that are \textbf{not necessary idempotent}. We are interested in studying Quantales since they semantically capture both Intuitionistic and Fuzzy Logic, so we will focus on the study of \textbf{commutative integral Quantales}. This kind of structures is widely use in the field of \textbf{substructural logics} as semantical counterparts for those logics, particularly Residuated Logic.

\begin{Def}[Commutative integral Quantales]\label{residuatedlattice}

We say that $\mathbb{Q}=(\mathbb{Q},\land,\lor,\cdot,\rightarrow, 1, 0)$ is a \textit{commutative integral Quantales} (or equivalently a \textit{complete residuated lattice}) if:

\begin{enumerate}
    \item $(\mathbb{Q},\land,\lor, 1, 0)$ is a complete bounded lattice.
     \item $(\mathbb{Q}, \cdot, 1)$ is a commutative monoid.
     \item For all $x, y_i\in \mathbb{Q}$ with $i\in I$,
    \begin{center}
        $x\cdot\bigvee\limits_{i\in I}y_i=\bigvee\limits_{i\in I}(x\cdot y_i)$  
    \end{center}
    and $\rightarrow$ can be defined as $x\rightarrow y:=\bigvee\{z\in\mathbb{Q}: x\cdot z\leq y\}$ 
\end{enumerate}
\end{Def}

\begin{Def}
    We say that $\mathbb{Q}$ is \textit{idempotent} if $x\cdot x=x$ for all $x\in\mathbb{Q}$.
\end{Def}

We now introduce the notion of $t$-norms, which are a key example, since they are a fundamental operation in the context of fuzzy logics. Here $[0, 1]$ denotes the subset of real numbers between 0 and 1.

\begin{Def}\label{tnorm}
        A function $\cdot: [0,1]^2\rightarrow [0,1]$ is called \textit{$t$-norm} if for all $x, y, a, b\in [0,1]$:
        \begin{enumerate}
            \item Commutativity: $x\cdot y=y\cdot x$.
            \item Associativity: $(x \cdot y) \cdot z = x \cdot (y \cdot z)$.
            \item Monotonicity: If $x\leq a$ and $y\leq b$, then $x\cdot y\leq a\cdot b$.
            \item Identity: $x\cdot 1= x$.
        \end{enumerate}
\end{Def}

\begin{Def}\label{continuoustnorms}
    Let $\cdot: [0,1]^2\rightarrow [0,1]$ be a $t$-norm. Then, $\cdot$ is said to be
    \begin{enumerate}
        \item \textit{left continuous}, if it is left continuous as a function from $[0,1]^2$ to $[0,1]$ with the usual metric.
        \item \textit{continuous}, if it is continuous as a function from $[0,1]^2$ to $[0,1]$ with the usual metric.
    \end{enumerate}
\end{Def}

\begin{Ej}\label{examplesleftcontinoustnorms}
    The following operations are left continuous $t$-norms:
    \begin{enumerate}
        \item The Łukasiewicz $t$-norm: $x\cdot_L y=max\{x+y-1, 0\}$.
        \item The product $t$-norm: $x\cdot_p y=x\cdot y$, where $\cdot$ denotes the usual product on $\mathbb{R}$.
        \item The Gödel-Dummett $t$-norm: $x\cdot_{GD} y=min\{x, y\}$.
    \end{enumerate}
\end{Ej}

\begin{Ej}The following structures are commutative integral Quantales:
    \begin{enumerate}
        \item Boolean algebras.
        \item Heyting algebras.
        \item The order $([0, 1], \leq)$ endowed with the $t$-norm of Łukasiewicz, Gödel-Dummett or the product $t$-norm.
        \item More generally, every structure $([0, 1], \leq \land,\lor,\cdot, 0, 1)$, where $\leq$ is the usual order and $\cdot$ is any left continuous t-norm.  
\end{enumerate}
    
\end{Ej}

As we mentioned before, we focus on integral commutative Quantales, since these structures naturally generalize both to Heyting algebra (Intuitionistic Logic) and $[0, 1]$ endowed with some left continuous $t$-norm (Fuzzy Logics).

\begin{Teo}[\cite{BuPi2014}, p. 2]\label{cuantales1}
    Let $\mathbb{Q}$ be a commutative integral Quantale and $x, y, z \in \mathbb{Q}$ with $i\in I$. Then:
    \begin{enumerate}
        \item $x\leq y$ if and only if $(x\rightarrow y)=1$.
        \item $x\cdot(x\rightarrow y)\leq y$.
        \item $(1\rightarrow y)=y$
        \item $0=x\cdot 0=0\cdot x$.
        \item $(0\rightarrow y)=1$ 
        \item If $x\leq y$, then $x\cdot z\leq y\cdot z$.
        \item  $x\cdot y\leq x\land y$.
        \item If $x\leq y$, then $y\rightarrow z\leq x\rightarrow z$.
        \item If $x\leq y$, then $z\rightarrow x\leq z\rightarrow y$.
        \item $(x\cdot y)\rightarrow z= x\rightarrow(y\rightarrow z)$.
    \end{enumerate}
\end{Teo}

\begin{Def}\label{otheroperationsonQ}
    Let $\mathbb{Q}$ be a commutative integral Quantale and $x, y\in \mathbb{Q}$. We define
    \begin{enumerate}
        \item $\sim x:=x\rightarrow 0$ (negation),
        \item $x\equiv y:= (x\rightarrow y)\cdot(y\rightarrow x)$ (equivalence),
        \item $x^0=1$ and $x^{n+1}=x\cdot x^n$ for $x\in\mathbb{N}$ (exponentiation).
    \end{enumerate}
\end{Def}


\begin{Teo}[\cite{BuPi2014}, p. 2]\label{cuantales2}
    Let $\mathbb{Q}=(\mathbb{Q},\land,\lor,\cdot,\rightarrow, 1, 0)$ be a commutative integral Quantale and let $x, y, y_i, x_i \in \mathbb{Q}$ for $i\in I$. Then:
    \begin{enumerate}
        \item $x\cdot (\sim x)=0$, but in general it is not true that $x\lor \sim x=1$.
         \item $x\leq (\sim \sim x)$, but in general it is not true that $\sim \sim x\leq x$.
         \item $\sim (x\lor y)=(\sim x)\cdot (\sim y)$ (De Morgan's Law), but it is not generally true that $\sim (x\cdot y)=(\sim x)\lor (\sim y)$.
         \item If $x\leq y$, then $(\sim y)\leq (\sim x)$ and $(\sim\sim x)\leq (\sim\sim y)$.
         \item $\sim 0=1$ and $\sim 1=0$.
        \item $x=y$ if and only if $(x\equiv y)=1$.
        \item $(\sim\sim x)\cdot (\sim\sim y)\leq (\sim\sim(x\cdot y))$.
        \item $\sim\sim\sim x=\sim x$.
    \end{enumerate}
\end{Teo}

If we consider a commutative integral Quantale that is also \textbf{ idempotent}, then the structure  collapses to a Heyting algebra.
For this reason, we are interested in studying commutative integral Quantales that are not necessarily idempotent.

\begin{Teo}[Folklore]\label{idempotentisheyting}
    If $\mathbb{Q}$ is a commutative idempotent integral Quantale, then $\mathbb{Q}$ is a Heyting algebra, where $x\cdot y=x\land y$ for all $x, y\in \mathbb{Q}$. 
\end{Teo}
    
\begin{Teo}\label{cerraduraBooleana}
    Let $\mathbb{Q}=(\mathbb{Q},\land,\lor,\cdot,\rightarrow, 1, 0)$ be a commutative integral Quantale. Then, if $x, y\in \{0, 1\}$, 
    \begin{enumerate}
        \item $x\rightarrow y\in \{0, 1\}$
        \item $x\land y\in \{0, 1\}$
        \item $x\lor y\in \{0, 1\}$
        \item $x\cdot y\in \{0, 1\}$
        \item $\sim x\in \{0,1\}$.
    \end{enumerate}
\end{Teo}    
  
\begin{proof}
    \begin{enumerate}
        \item By  Theorem \ref{cuantales1} items 3. and 4., $(0\rightarrow y)=1$ and $(1\rightarrow y)=y$ for all $y\in \{0, 1\}$, and this implies that $x\rightarrow y \in\{0, 1\}$ for all $x, y\in \{0, 1\}$.
        
         \item Since $0$ is the minimum and  $1$ the maximum of $\mathbb{Q}$, we have $x\land y=1$ if $x=y=1$ and $x\land y=0$ if any of them is equal to $0$.
         
         \item It is proved in a similar way as the previous item.
         
         \item Since $1$ is the module of the monoid ($\mathbb{Q}, \cdot, 1$), we have $x\cdot 1= 1\cdot x=x$ for all $x\in \{0, 1\}$, but on the other hand, by Theorem \ref{cuantales1} 6., we have $x\cdot 0=0\cdot x=0$ for all $x\in \{0, 1\}$. The above implies that $x\cdot y\in \{0, 1\}$ for all $x, y\in \{0, 1\}$.
        \item By Theorem \ref{cuantales2} 5., we have $\sim 0=1$ and $\sim 1=0$, which implies that $\sim x\in \{0,1\}$ for all $x, y\in \{0, 1\}$.
        
    \end{enumerate}
\end{proof}

\begin{Coro}
    If $\mathbb{Q}=(\mathbb{Q},\land,\lor,\cdot,\rightarrow, 1, 0)$ is a commutative integral Quantale, then $\{0, 1\}\subseteq \mathbb{Q}$ is a Boolean algebra with the operations inherited from $\mathbb{Q}$ and $x\cdot y=x\land y$ for all $x, y\in \{0, 1\}$.
\end{Coro}

\section{Residuated Logic}\label{residuatedlogic}

This logic was introduced by Ulrich Höhle \cite{Hohle1994} under the name {\it Monoidal Logic} in order to present a general framework for the study of {\it Fuzzy Logics} based on t-norms, {\it Intuitionistic Logic} and {\it Girard’s Linear Logic}. In his article, Höhle considers residuated integral commutative 1-monoids (i.e. complete residuated lattices in our terms) as a set of truth values of his logic, presents a {\it completeness} and {\it soundness theorem}, and shows some interactions of it with the other logics mentioned.

Throughout this work, we will call Höhle’s {\it Monoidal Logic} as {\it Residuated Logic}, following Lano’s notation \cite{Lano1992a} in his study of Residuated Logic and Fuzzy sets, where this logic is studied in its modal variant and is applied in the context of set-theoretic models valued on Residuated lattices. \\
    
    \begin{Def}[Logical symbols, \cite{Lano1992a}]\label{StrongConjunction}
        The fundamental difference between Classical (or Intuitionistic) Logic and Fuzzy (or Residuated) Logic is that we consider additional logical symbols, namely, in Residuated Logic, we consider two types of \textbf{conjunction}, a \textbf{weak conjunction} ($\land$) and a \textbf{strong conjunction} ($\&$). 
        
        The following symbols are definable, using the usual logical symbols:
        \begin{enumerate}
            \item Equivalence $\varphi \equiv \psi:=(\varphi \rightarrow \psi) \& (\psi \rightarrow \varphi)$,
            \item Negation $\sim \varphi := \varphi \rightarrow \bot$,
            \item Tautology $\top:=\sim \bot$.
        \end{enumerate}
    \end{Def}
    \begin{Def}[Propositional formulas, \cite{Lano1992a}]\label{residuatedformulas}
        The construction of formulas is done by recursion on a manner analogous to how it is done in Classical Propositional Logic. To differentiate these formulas from formulas in Classical (or Intuitionistic)  Logic, we call them \textit{Residuated (Propositional) formulas}, or $R$-formulas for short. 
    \end{Def}


\subsection{Models of Residuated Logic}

This logic has natural models in the form of \textit{Quantale-valued models}. In this models, the truth of the formulas are values within a complete commutative Quantale. We will not develop the theory of this models. The development of these models can be found in \cite{Lano1992a}, where Lano mentions the Completeness and Soundness of these models (where He calls them \textbf{Residuated
algebra valued models}) and then goes on to prove the Completeness Theorem for \textbf{topological Residuated algebra valued models} and \textbf{Residuated Modal Logic $RS_5$}. 

\section{Valued-models of sets theory.}\label{valuedmodelsofsetheory}
Throughout this section, $\mathbb{Q}$ is taken as a commutative integral Quantale.

In this section, we will discuss a way to build a version of the von Neumann universe in the context of commutative integral Quantale-valued models. 

    \begin{Def}\label{V^Q}
        We define $V_{\alpha}^{\mathbb{Q}}$ by recursion on ordinals:    
        \begin{enumerate}
            \item $V_0^{\mathbb{Q}}:=\emptyset$
            \item $V_{\alpha+1}^{\mathbb{Q}}:=\{f: f \text{ is a function with }dom(f)\subseteq V_{\alpha}^{\mathbb{Q}} \text{ and }ran(f)\subseteq \mathbb{Q}\}$
             \item $V_{\alpha}^{\mathbb{Q}}:=\bigcup\limits_{\beta<\alpha} V_{\beta}^{\mathbb{Q}}$ with $\alpha\not =0$ limit ordinal.
            \item $V^{\mathbb{Q}}:=\bigcup\limits_{\alpha\in ON} V_{\alpha}^{\mathbb{Q}}$.        
        \end{enumerate}
    \end{Def}

     \begin{Def}
        We define the language $\mathcal{L}_{\mathbb{Q}}$ as $\mathcal{L}_{\mathbb{Q}}=\mathcal{L}_{\in}\cup \{c_a: a\in V^{\mathbb{Q}}\}$, where each $c_a$ is a constant symbol.
    \end{Def}

    \begin{Def}
        We define interpretations of $\in$ and $=$ in $V^{\mathbb{Q}}$ as  
        \begin{enumerate}
            \item $\ldbrack  f \subseteq g  \rdbrack_{\mathbb{Q}}:=\bigwedge\limits_{x\in dom( f )}( f (x)\rightarrow\ldbrack x\in  g \rdbrack_{\mathbb{Q}})
            $
            
            \item $\ldbrack  f = g  \rdbrack_{\mathbb{Q}}:=\ldbrack  f \subseteq g  \rdbrack_{\mathbb{Q}}\cdot \ldbrack  g \subseteq f  \rdbrack_{\mathbb{Q}}
            $
            
            \item $\ldbrack  f \in g  \rdbrack_{\mathbb{Q}}:=\bigvee\limits_{x\in dom( g )}( g (x)\cdot \ldbrack x=  f \rdbrack_{\mathbb{Q}})$            
        \end{enumerate}
    \end{Def}

\begin{Def}\label{valuacionquantal}
        Since we already have a valuation for atomic sentences, we extend this to a valuation on all $R$-sentences in the language $\mathcal{L}_{\mathbb{Q}}$ by recursion on the complexity of the sentences:
        \begin{enumerate}
            \item Let $\psi$ and $\varphi$ be $R-\mathcal{L}_{\mathbb{Q}}$-sentences and $\theta(x)$ an $R-\mathcal{L}_{\mathbb{Q}}$-formula with free variable $x$.
                         
            \begin{enumerate}[label=\alph*.]
            
            \item $\ldbrack \psi\& \varphi \rdbrack_{\mathbb{Q}}:=\ldbrack\psi \rdbrack_{\mathbb{Q}}\cdot\ldbrack\varphi \rdbrack_{\mathbb{Q}}$
            
            \item $\ldbrack \psi\rightarrow\varphi \rdbrack_{\mathbb{Q}}:=\ldbrack\psi \rdbrack_{\mathbb{Q}}\rightarrow\ldbrack\varphi \rdbrack_{\mathbb{Q}}$
            
            \item $\ldbrack \psi\land \varphi \rdbrack_{\mathbb{Q}}:=\ldbrack\psi \rdbrack_{\mathbb{Q}}\land\ldbrack\varphi \rdbrack_{\mathbb{Q}}$
            
            \item $\ldbrack \psi\lor \varphi \rdbrack_{\mathbb{Q}}:=\ldbrack\psi \rdbrack_{\mathbb{Q}}\lor\ldbrack\varphi \rdbrack_{\mathbb{Q}}$
            
            \item $\ldbrack \exists x\theta(x) \rdbrack_{\mathbb{Q}}:=\bigvee\limits_{a\in V^{\mathbb{Q}}}\ldbrack\theta(c_a) \rdbrack_{\mathbb{Q}}$
            
            \item $\ldbrack \forall x\theta(x) \rdbrack_{\mathbb{Q}}:=\bigwedge\limits_{a\in V^{\mathbb{Q}}}\ldbrack\theta(c_a) \rdbrack_{\mathbb{Q}}$
            \end{enumerate}     
        \end{enumerate}

\end{Def}  

\begin{Def}
    If $\varphi$ is an $R-\mathcal{L}_{\mathbb{Q}}$-sentence, we say that $V^{\mathbb{Q}}$ is a model of $\varphi$ if $\ldbrack \varphi \rdbrack_{\mathbb{Q}}=1$.
\end{Def}

    \begin{Teo}[\cite{SS1967, Grayson1979, Lano1992a}]
    \begin{enumerate}
        \item If $\mathbb{B}$ is a Boolean algebra, then $V^{\mathbb{B}}$ is a model of Classical Logic.
        \item If $\mathbb{H}$ is a Heyting algebra, then $V^{\mathbb{H}}$ is a model of Intuitionistic Logic.
        \item If $\mathbb{Q}$ is a complete residuated lattice, then $V^{\mathbb{Q}}$ is a model of Residuated Logic.
    \end{enumerate}        
    \end{Teo}

\subsection{Constructibility on valued-models of Set Theory}\label{constructiblesetsonvaluedmodels}

In this subsection, we will study two different versions of the constructible sets G\"odel universe in the context of commutative integral Quantale-valued models (Definitions~\ref{mathfrakL} -by using a weak notion of definibility- and \ref{mathbbL}). We will prove that they are isomorphic to the classical von Neumann and G\"odel universes, respectively.

\begin{Def}[\cite{Moncayo2023}]
        Let $M\subseteq V^{\mathbb{Q}}$ be a subclass. We can view $M$ as a $\mathbb{Q}$-valued $\mathcal{L}_{\in}$-structure by taking the restrictions on $M$ from the interpretations of $\in$ and $=$ on $V^{\mathbb{Q}}$ and we can obtain a valuation $\ldbrack \cdot\rdbrack_M$.\\
        We say that $f$ is a \textit{weakly }$\mathbb{Q}-$\textit{definable subset} of $M$ if:
        \begin{enumerate}
            \item $f\in V^\mathbb{Q}$.
            \item $dom(f)\subseteq M$.
            \item There is a first-order $\mathcal{L}_{\in}$-formula $\varphi(x,\bar{y})$ and $\bar{b}\in M^{|\bar{y}|}$ such that for every $a\in dom(f)$
            \begin{center}
                $f(a)=\ldbrack \varphi(a,\bar{b})\rdbrack_M$.
            \end{center}
        \end{enumerate}
        And we define the class of weakly $\mathbb{Q}$-definable subsets of $M$ as
        \begin{center}
            $Def^{\mathbb{Q}^*}(M):=\{f\in V^{\mathbb{Q}} :f \text{ is a weakly }\mathbb{Q}\text{-definable subset of }M\}.$
        \end{center}
\end{Def}

\begin{Def}[\cite{Moncayo2023}]\label{mathfrakL}
        We define $\mathfrak{L}_{\alpha}^{\mathbb{Q}}$, by transfinite recursion over the ordinals, as follows
        
        \begin{enumerate}
            \item $\mathfrak{L}_0^{\mathbb{Q}}:=\emptyset$.
            \item $\mathfrak{L}_{\alpha+1}^{\mathbb{Q}}:=Def^{\mathbb{Q}^*}(\mathfrak{L}_{\alpha}^{\mathbb{Q}}).$ 
            \item $\mathfrak{L}_{\alpha}^{\mathbb{Q}}:=\bigcup\limits_{\beta<\alpha} \mathfrak{L}_{\beta}^{\mathbb{Q}}$ for $\alpha\not=0$ limit ordinal.
            \item $\mathfrak{L}^{\mathbb{Q}}:=\bigcup\limits_{\alpha\in ON} \mathfrak{L}_{\alpha}^{\mathbb{Q}}$.
        \end{enumerate}
\end{Def}

\begin{Prop}[\cite{Moncayo2023}]
    For all  $\alpha, \beta\in ON$, 
    \begin{enumerate}
        \item $\mathfrak{L}_{\alpha}^{\mathbb{Q}}\subseteq V^{\mathbb{Q}}_{\alpha}$.
        \item $\mathfrak{L}_{\alpha}^{\mathbb{Q}}\subseteq \mathfrak{L}^{\mathbb{Q}}_{\alpha+1}$.
        \item If $\alpha<\beta$, then $\mathfrak{L}_{\alpha}^{\mathbb{Q}}\subseteq \mathfrak{L}^{\mathbb{Q}}_{\beta}$
        
    \end{enumerate}
    
\end{Prop}

\begin{Teo}[\cite{Moncayo2023}]\label{L^Bistwovalued}
    If $f\in \mathfrak{L}^{\mathbb{Q}}$, then $ran(f)\subseteq \{0,1\}=2$. That is, $\mathfrak{L}^{\mathbb{Q}}\subseteq V^2$. 
\end{Teo}

\begin{proof}
    We prove it by induction on ordinals, showing that for all  $\alpha\in ON$, if $f\in \mathfrak{L}_{\alpha}^{\mathbb{Q}}$, then $ran(f)\subseteq 2$.
    
    \textbf{Induction hypothesis $1$}: Take $\alpha\in ON$ such that for all $a\in \mathfrak{L}_{\alpha}^{\mathbb{Q}}$, $ran(a)\subseteq 2$.
    
    Let us see that if $f\in \mathfrak{L}^{\mathbb{Q}}_{\alpha+1}$, then $ran(f)\subseteq \{0,1\}=2$. Since $f\in \mathfrak{L}^{\mathbb{Q}}_{\alpha+1}$, there is a $\mathcal{L}_{\in}$-formula $\varphi(x, \bar{y})$ with $|\bar{y}|=n$ and $\bar{b}\in (\mathfrak{L}^{\mathbb{Q}}_{\alpha})^{n}$ such that for all  $a\in dom (f)\subseteq \mathfrak{L}^{\mathbb{Q}}_{\alpha}$
    \begin{center}
        $f(a)=\ldbrack \varphi(a,\bar{b})\rdbrack_{\mathfrak{L}_{\alpha}^{\mathbb{Q}}}$.
    \end{center}
    Notice that $a,b_1, b_2, ..., b_n\in \mathfrak{L}^{\mathbb{Q}}_{\alpha}$, so we can use the induction hypothesis $1$ on them, so that $ran(a), ran(b_i)\subseteq 2$ for all  $1\leq i\leq n$.
    
    We will prove that if $f\in \mathfrak{L}^{\mathbb{Q}}_{\alpha+1}$, then $ran(f)\subseteq \{0,1\}=2$, using induction on formulas, where the formulas can take parameters from $\mathfrak{L}^{\mathbb{Q}}_{\alpha}$.
    
    We start with the atomic case. We want to prove that $\ldbrack a\in b\rdbrack, \ldbrack b\in a\rdbrack, \ldbrack a=b\rdbrack\in 2$ for all  $a, b\in\mathfrak{L}_{\alpha}^\mathbb{Q}$.

    We prove the statement given above by induction on the well-founded relation $<$ on $ V^{\mathbb{Q}}$, where
    \begin{center}
        $(v,w)<(a,b)$ if and only if $(v=a \text{ and } w\in dom(b)) \text{ or } (v\in dom(a) \text { and } w=b)$, where $a, b, f, g\in V^{\mathbb{Q}}$.
    \end{center}
    
    \textbf{Induction hypothesis 2}: for all  $<$-predecessors $(v,w)$ of $(f, g)$, 
     \begin{center}
         $\ldbrack v\in w\rdbrack, \ldbrack w\in v\rdbrack, \ldbrack v=w\rdbrack\in 2$.
     \end{center}
     Notice that the $<$-predecessors $(v,w)$ of $(a, b)$ have the form
    \begin{center}
        $(a, w)$ and $(v, b)$, where $v\in dom (a)$ and $w\in dom(b)$.
    \end{center}
    Then, the induction hypothesis $2$ means that for all  $v\in dom (a)$ and $w\in dom(b)$, 
    \begin{center}
        $\ldbrack a\in w\rdbrack, \ldbrack w\in a\rdbrack, \ldbrack v\in b\rdbrack, \ldbrack b\in v\rdbrack,\ldbrack a=w\rdbrack, \ldbrack v=b\rdbrack\in 2$.
    \end{center}
    
    Let us see then that
    \begin{center}
        $\ldbrack a\in b\rdbrack, \ldbrack b\in a\rdbrack, \ldbrack a=b\rdbrack\in 2$
    \end{center}
    By definition of $\ldbrack \cdot\in \cdot\rdbrack_{\mathfrak{L}_{\alpha}^{\mathbb{Q}}}$, we have that
    
    \begin{center}
        $\ldbrack a\in b\rdbrack_{\mathfrak{L}_{\alpha}^{\mathbb{Q}}}=\ldbrack a\in b\rdbrack_{V}=\bigvee\limits_{w\in dom(b)}b(w)\land \ldbrack a=w\rdbrack$
    \end{center}
    Given that $b\in \mathfrak{L}_{\alpha}^{\mathbb{Q}}$, by  induction hypothesis $1$, we have that $b(w)\in 2$. On the other hand, the induction hypothesis $2$ implies that $\ldbrack a=w\rdbrack\in 2$, so $b(w)\land\ldbrack a=w\rdbrack\in 2$ and $\bigvee\limits_{w\in dom(b)}b(w)\land \ldbrack a=w\rdbrack=\ldbrack a\in b\rdbrack_{\mathfrak{L}_{\alpha}^{\mathbb{Q}}}\in 2$.
    
    In a similar way as before, it is shown that $\ldbrack b\in a\rdbrack_{\mathfrak{L}_{\alpha}^{\mathbb{Q}}}\in 2$
    
    By definition of $\ldbrack \cdot= \cdot\rdbrack_{\mathfrak{L}_{\alpha}^{\mathbb{Q}}}$, we have that
    
    \begin{center}
        $\ldbrack a= b\rdbrack_{\mathfrak{L}_{\alpha}^{\mathbb{Q}}}=\ldbrack a=b\rdbrack_{V}=\left(\bigwedge\limits_{v\in dom(a)}a(v)\rightarrow\ldbrack v\in b\rdbrack\right)\land\left(\bigwedge\limits_{w\in dom(b)}b(w)\rightarrow\ldbrack w\in a\rdbrack\right)$
    \end{center}
    
    Since $a, b\in \mathfrak{L}_{\alpha}^{\mathbb{Q}}$, by  induction hypothesis $1$, we have that $a(v), b(w)\in 2 $. On the other hand, the induction hypothesis $2$ implies that $\ldbrack v\in b\rdbrack, \ldbrack w\in a\rdbrack\in 2$, thus
    \begin{center}
        $(a(v)\rightarrow\ldbrack v\in b\rdbrack), (b(w)\rightarrow\ldbrack w\in a\rdbrack)\in 2$.
    \end{center}
    
    Then,
    
    \begin{center}
        $\bigwedge\limits_{v\in dom(a)}a(v)\rightarrow\ldbrack v\in b\rdbrack, \bigwedge\limits_{w\in dom(b)}b(w)\rightarrow\ldbrack w\in a\rdbrack\in 2$.
    \end{center}
    
    Therefore,
    
    \begin{center}
        $\left(\bigwedge\limits_{v\in dom(a)}a(v)\rightarrow\ldbrack v\in b\rdbrack\land\bigwedge\limits_{w\in dom(b)}b(w)\rightarrow\ldbrack w\in a\rdbrack\right)=\ldbrack a= b\rdbrack_{\mathfrak{L}_{\alpha}^{\mathbb{Q}}}\in 2$.
    \end{center}
    
    Thus, by induction on the well-founded relation $<$, we have that $\ldbrack a\in b\rdbrack, \ldbrack b\in a\rdbrack, \ldbrack a=b\rdbrack\in 2$ for all  $a , b\in \mathfrak{L}_{\alpha}^{\mathbb{Q}}$. this proves the case for atomic $\mathcal{L}_{\in}$-formulas.
    
    By induction on $\mathcal{L}_{\in}$-formulas, it is straightforward to show that for every $\mathcal{L}_{\in}$-formula $\varphi(x, \bar{y})$ with $|y|=n$, $\bar {b}\in (\mathfrak{L}^{\mathbb{Q}}_{\alpha})^{n}$y $a\in dom (f)\subseteq \mathfrak{L}^{\mathbb {B}}_{\alpha}$ we have that
    \begin{center}
        $\ldbrack \varphi(a,\bar{b})\rdbrack_{\mathfrak{L}_{\alpha}^{\mathbb{Q}}}
        =f(a)\in 2$,
    \end{center}
    since if Boolean combinations of formulas that only take values in $\{0, 1\}$ are made, the result of evaluating these formulas is $0$ or $1$.
    
    Thus, we have that, for all  $f\in \mathfrak{L}^{\mathbb{Q}}_{\alpha+1}$, $ran(f)\subseteq 2$.
    
    Checking the limit ordinal case is straightforward. Then, by induction on the ordinals, we have that for all  $f\in \mathfrak{L}^{\mathbb{Q}}$, $ran(f)\subseteq 2$.
\end{proof}

This theorem tells us that the logic that governs these models is \textbf{bi-valued}.

\begin{Def}\label{inmersiongorro}
    Take $\mathbb{Q}=2$. We define a class function $\hat{\cdot}:V\rightarrow V^{2}$ as follows: Given $x\in V$, take
     \begin{center}
         $\hat{x}=\{(\hat{y},1):y\in x\}$.
     \end{center}
     This is a definition by recursion on the well-founded relation $y\in x$. Notice that, for all  $x\in V$,
     \begin{center}
         $\hat{x}\in V^2\subseteq V^{\mathbb{Q}}$.
     \end{center}
\end{Def}

\begin{Teo}[cf. \cite{Bell2005}, Theorem 1.23]\label{functionhat}
    Let $x, y, a_1, ..., a_n\in V$ and $\varphi(x_1, ..., x_n)$ be a $\mathcal{L}_{\in}$-formula. Then
    \begin{enumerate}
        \item $\ldbrack \hat{x}=\hat{y}\rdbrack_{\mathbb{Q}}=\ldbrack \hat{x}=\hat{y}\rdbrack_{\mathbb{B}}=\ldbrack \hat{x}=\hat{y} \rdbrack_{2}\in 2$.
        \item $\ldbrack \hat{x}\in\hat{y}\rdbrack_{\mathbb{Q}}=\ldbrack \hat{x}=\hat{y}\rdbrack_{\mathbb{B}}=\ldbrack \hat{x}\in\hat{y} \rdbrack_{2}\in 2$.
        \item $\ldbrack \varphi(\hat {a_1}, ..., \hat {a_n})\rdbrack_{\mathbb{Q}}, \ldbrack \varphi(\hat{a_1}, ..., \hat{a_n}) \rdbrack_{2}\in 2$
        \item $\hat{\cdot}$ is injective.
        \item $\hat{\cdot}$ is surjective in the following way: For every $u\in V^2$ there exists a \textbf{unique} $a\in V$ such that $V^{\mathbb{Q}}\models u=\hat{a}$.
        \item $\hat{\cdot}$ is an ``isomorphism'' in the following way 
        \begin{center}
            $\varphi(a_1, ..., a_n)$ holds in $V$, if and only if, $\ldbrack \varphi(\hat{a_1}, ..., \hat{a_n})\rdbrack_{V^2}=1$
        \end{center}
        \item If $\varphi(x_1, ..., x_n)$ is an $\mathcal{L}_{\in}$-formula with bounded quantifiers (i.e. if each of its quantifiers occurs in the
form $\forall x\in a$ or $\exists x\in b$)
        \begin{center}
            $\varphi(a_1, ..., a_n)$ holds in $V$, if and only if, $\ldbrack \varphi(\hat{a_1}, ..., \hat{a_n})\rdbrack_{V^\mathbb{Q}}=1$
        \end{center}
    \end{enumerate}
    
\end{Teo}

\begin{Teo}[\cite{Moncayo2023}]
    For all  $x\in V$, $\hat{x}\in \mathfrak{L}^{\mathbb{Q}}$. 
\end{Teo}

Therefore, since $\mathfrak{L}^{\mathbb{Q}}\subseteq V^2$, $\mathfrak{L}^{\mathbb{Q}}$ is essentially $V$.

We then proceed to change our definition to see if we can get a more interesting model.

\begin{Def}[\cite{Moncayo2023}]
        Let $M\subseteq V^{\mathbb{Q}}$. We say that $f\in V^\mathbb{Q}$ is a $\mathbb{Q}-$\textit{definable} subset of $M$ if $f$ satisfies the following
         \begin{enumerate}
             \item $dom(f)=M$ (before it was $dom(f)\subseteq M$). 
             \item There is a classical first-order $\mathcal{L}_{\in}$-formula $\varphi(x,\bar{y})$ and $\bar{b}\in M^{|\bar{y}| }$ such that for all  $a\in dom(f)$
             \begin{center}
                 $f(a)=\ldbrack \varphi(a,\bar{b})\rdbrack_M$
             \end{center}
         \end{enumerate}
         And we define the set of $\mathbb{Q}$-definable subsets of $M$ as
        \begin{center}
            $Def^{\mathbb{Q}}(M):=\{f\in V^{\mathbb{Q}} :f \text{ is a }\mathbb{Q}\text{-definable subset of }M\}.$
        \end{center}
\end{Def}

\begin{Def}[\cite{Moncayo2023}]\label{mathbbL}
        We define by transfinite recursion over the ordinals

        \begin{enumerate}
            \item $\mathbb{L}_0^{\mathbb{Q}}:=\emptyset$
            \item $\mathbb{L}_{\alpha+1}^{\mathbb{Q}}:=Def^{\mathbb{Q}}(\mathbb{L}_{\alpha}^{\mathbb{Q}})\cup \mathbb{L}_{\alpha}^{\mathbb{Q}}$
            \item $\mathbb{L}_{\alpha}^{\mathbb{Q}}:=\bigcup\limits_{\beta<\alpha} \mathbb{L}_{\beta}^{\mathbb{Q}}$ for $\alpha\not=$0 limit.
            \item $\mathbb{L}^{\mathbb{Q}}:=\bigcup\limits_{\alpha\in ON} \mathbb{L}_{\alpha}^{\mathbb{Q}}$.
        \end{enumerate}     
\end{Def}

\begin{Obs}
    We can prove that $\mathbb{L}^{\mathbb{Q}}\subseteq V^2$ mimicking the argument that we used for the model $\mathfrak{L}^{\mathbb {B}}$ in Theorem \ref{L^Bistwovalued}.
\end{Obs}

\begin{Lema}\label{lemaigualdad}
    For all  $f, g\in \mathbb{L}^{\mathbb{Q}}$, if $\ldbrack f=g \rdbrack_{\mathbb{L}^{\mathbb{Q}}} =1$, then, for all  $\mathcal{L}_{\in}$-formula $\varphi(x, \bar{y})$ and $\bar{a}\in M^{|\bar{y}|}$,
    \begin{center}
        $\ldbrack \varphi(f, \bar{a})\rdbrack_{\mathbb{L}^{\mathbb{Q}}}=\ldbrack \varphi(g, \bar{a})\rdbrack_{\mathbb{L}^{\mathbb{Q}}}.$
    \end{center}
\end{Lema}

\begin{Lema}\label{lemaextensionigualdad}
    Let $f, g\in \mathbb{L}^{\mathbb{Q}}$. Suppose $f$ is an extension of $g$, i.e. $dom(g)\subseteq dom(f)$ and $f\restriction_{dom(g)}=g$. If $f(a)=0$ for all  $a\in dom(f)\setminus dom(g)$, then $V^{\mathbb{Q}}\models f=g$.
\end{Lema}

The following theorem shows us that, for all Boolean algebras $\mathbb{Q}$, $L$ is ``isomorphic'' to $\mathbb{L}^{\mathbb{Q}}$ in the following way:

\begin{Teo}[\cite{Moncayo2023}]
    There exists a class function $j:L\rightarrow \mathbb{L}^{\mathbb{Q}}$ such that for all  $\alpha\in ON$, the restriction $j\restriction_{L_{\alpha}}: L_{\alpha}\rightarrow \mathbb{L}^{\mathbb{Q}}$ satisfies
    \begin{enumerate}

        \item $ran(j\restriction_{L_{\alpha}})\subseteq \mathbb{L}_{\alpha}^{\mathbb{Q}}$.
        \item $j\restriction_{L_{\alpha}}$ is injective.         

        \item $j\restriction_{L_{\alpha}}$ is surjective in the following sense: for all  $Y\in \mathbb{L}_{\alpha}^{\mathbb{Q}}$, there exists $X\in L_{\alpha}$ such that
        \begin{center}
            $\mathbb{L}_{\alpha}^{\mathbb{Q}}\models j(X)=Y$.
        \end{center}
        \item $j\restriction_{L_{\alpha}}$ is an elementary embedding in the following sense: for every $\mathcal{L}_{\in}$-formula $\varphi(\bar{x})$ and $\bar{a}\in L_{\alpha}^{|\bar{x}|}$, 
        \begin{center}
            $L_{\alpha}\models \varphi(\bar{a})$ if and only if $\ldbrack\varphi(j(\bar{a}))\rdbrack_{\mathbb{L}_{\alpha}^{\mathbb{Q}}}=1$.
        \end{center}
    \end{enumerate}
\end{Teo}

\begin{proof}
    We prove this by induction on ordinals:
    The case $\alpha=0$ is trivial.
    
    Suppose that we have already defined $j\restriction_{L_{\beta}}$ and that it satisfies the conditions of the theorem for all  $\beta\leq \alpha$, where $\alpha$ is an ordinal. We define $j$ for $L_{\alpha+1}\setminus L_{\alpha}$: Given $X\in L_{\alpha+1}\setminus L_{\alpha}$, we have that $X\subseteq L_{\alpha}$ and that there exists a first-order $\mathcal{L}_{\in}$-formula $\varphi(x, \bar{y})$ and $\bar{b}\in L_{\alpha}^{|\bar{y}|}$ such that

    \begin{center}
        $X=\{a\in L_{\alpha}: L_{\alpha}\models \varphi(a, \bar{b})\}$.
    \end{center}
    
    We define
    
    \begin{center}
        $j(X):\mathbb{L}_{\alpha}^{\mathbb{Q}}\rightarrow \mathbb{Q}$ as $j(X)(c)=\ldbrack \varphi(c, j(\bar{b}))\rdbrack_{\mathbb{L}_{\alpha}^{\mathbb{Q}}}$ for all  $c\in \mathbb{L}_{\alpha}^{\mathbb{Q}}$
    \end{center}

    Notice that, from the definition, $j(X)\in \mathbb{L}_{\alpha+1}\setminus \mathbb{L}_{\alpha}^{\mathbb{Q}}$, and thus $ran(j\restriction_{ L_{\alpha+1}})\subseteq \mathbb{L}_{\alpha+1}$ and $rank_L(X)=rank_{\mathbb{L}^{\mathbb{Q}}}(j(X))$. 
    
    Let us see that the function $j(X)$ is indeed well-defined, that is, that the function does not depend on the choice of formulas and parameters.\\
    
    Consider some $\mathcal{L}_{\in}$-formulas $\varphi(x, \bar{y})$, $\psi(z, \bar{w})$ and parameters $\bar{b }\in L_{\alpha}^{|\bar{y}|}$ and $\bar{d}\in L_{\alpha}^{|\bar{w}|}$ such that

    \begin{center}
        $X=\{a\in L_{\alpha}: L_{\alpha}\models \varphi(a, \bar{b})\}=\{a\in L_{\alpha}: L_{\alpha}\models \psi(a, \bar{d})\}$.
    \end{center}

    From the previous equality, we have that for all  $a\in L_{\alpha}$,
    
    \begin{center}
        $ L_{\alpha}\models \varphi(a, \bar{b})$ if and only if $L_{\alpha}\models \psi(a, \bar{d})$.
    \end{center}
    
    But using the induction hypothesis item $4.$, we have that
    
    \begin{center}
        $L_{\alpha}\models \varphi(a, \bar{b})$ if and only if $\ldbrack\varphi(j(a), j(\bar{b}))\rdbrack_{\mathbb{L}_{\alpha}^{\mathbb{Q}}}=1$ and \\
        $L_{\alpha}\models \psi(a, \bar{d})$ if and only if $\ldbrack\psi(j(a), j(\bar{d}))\rdbrack_{\mathbb{L}_{\alpha}^{\mathbb{Q}}}=1$
    \end{center}
    
    Combining these results, we get
    
    \begin{center}
        $\ldbrack\varphi(j(a), j(\bar{b}))\rdbrack_{\mathbb{L}_{\alpha}^{\mathbb{Q}}}=1$ if and only if $\ldbrack\psi(j(a), j(\bar{d}))\rdbrack_{\mathbb{L}_{\alpha}^{\mathbb{Q}}}=1$, namely,\\
        
        $\ldbrack\varphi(j(a), j(\bar{b}))\rdbrack_{\mathbb{L}_{\alpha}^{\mathbb{Q}}}=\ldbrack\psi(j(a), j(\bar{d}))\rdbrack_{\mathbb{L}_{\alpha}^{\mathbb{Q}}}$ for all  $a\in L_{\alpha}$
    \end{center}
    
    And thus the function $j(X)$ is well-defined for every $c\in j(\mathbb{L}_{\alpha}^{\mathbb{Q}})\subseteq \mathbb{L}_{\alpha}^{\mathbb{Q}}$.
    
    To see that this is also true for all  $c\in \mathbb{L}_{\alpha}^{\mathbb{Q}}$, and not only for all $j(a)\in j(L_{\alpha})$, we use Lemma \ref{lemaigualdad} and the ``surjectivity'' of $j$. Given $c\in \mathbb{L}_{\alpha}^{\mathbb{Q}}$, there exists $a_c\in L_{\alpha}$ such that $\ldbrack j(a_c)=c\rdbrack_{\mathbb{L}_{\alpha}^{\mathbb{Q}}}=1$. Then,
    
    \begin{center}
        $\ldbrack\varphi(j(a_c), j(\bar{b}))\rdbrack_{\mathbb{L}_{\alpha}^{\mathbb{Q}}}=1$ if and only if $\ldbrack\varphi(c, j(\bar{b}))\rdbrack_{\mathbb{L}_{\alpha}^{\mathbb{Q}}}=1$ and $\ldbrack\psi(j(a_c), j(\bar{d}))\rdbrack_{\mathbb{L}_{\alpha}^{\mathbb{Q}}}=1$ if and only if $\ldbrack\psi(c, j(\bar{d}))\rdbrack_{\mathbb{L}_{\alpha}^{\mathbb{Q}}}=1$
    \end{center}
    
    we conclude that
    
    \begin{center}
        $\ldbrack\varphi(c, j(\bar{b}))\rdbrack_{\mathbb{L}_{\alpha}^{\mathbb{Q}}}=1$ if and only if $\ldbrack\psi(c, j(\bar{d}))\rdbrack_{\mathbb{L}_{\alpha}^{\mathbb{Q}}}=1$, i.e.,  $\ldbrack\varphi(c, j(\bar{b}))\rdbrack_{\mathbb{L}_{\alpha}^{\mathbb{Q}}}=\ldbrack\psi(c, j(\bar{d}))\rdbrack_{\mathbb{L}_{\alpha}^{\mathbb{Q}}}$ for all $c\in \mathbb{L}_{\alpha}^{\mathbb{Q}}$. 
    \end{center}
    And therefore the function $j(X)$ is well-defined for $X\in\mathbb{L}_{\alpha+1}^{\mathbb{Q}}$.\\
    
    Let us see that $j\restriction_{\mathbb{L}_{\alpha}^{\mathbb{Q}}}$ is injective.\\
    
     If $X, Y\in L_{\alpha}$, then, by induction hypothesis item 2., $j(X)\not=j(Y)$ provided that $X\not=Y$.\\
     
    If $X\in L_{\alpha+1}$ and $Y\in L_{\alpha}$, then we have that there exists $\beta<\alpha$ such that $Y\in L_{\beta+1 }$ and therefore $dom(X)=\mathbb{L}_{\alpha}^{\mathbb{Q}}$ and $dom(Y)=\mathbb{L}_{\beta}^{\mathbb{Q}}$, so we have $j(X) \not=j(Y)$, since $dom(j(X))\not=dom(j(Y))$.\\

     Let $X, Y\in L_{\alpha+1}\setminus L_{\alpha}$ be such that $X\not=Y$. Thus, there are $\mathcal{L}_{\in}$-formulas $\varphi(x, \bar{y})$, $\psi(z, \bar{w})$ and parameters $\bar{ b}\in L_{\alpha}^{|\bar{x}|}$ and $\bar{d}\in L_{\alpha}^{|\bar{w}|}$ such that

    \begin{center}
        $X=\{a\in L_{\alpha}: L_{\alpha}\models \varphi(a, \bar{b})\}$ and $Y=\{a\in L_{\alpha}: L_{\alpha}\models \psi(a, \bar{d})\}$.
    \end{center}
    
    Since $X\not= Y$, we may assume, without loss of generality, that there exists $a\in L_{\alpha}$ such that $a\in X$ and $a\not\in Y$, i.e.
    
    \begin{center}
        $ L_{\alpha}\models \varphi(a, \bar{b})$ and $L_{\alpha}\nvDash \psi(a, \bar{d})$.
    \end{center}
    
    Then,
    
    \begin{center}
        $j(X)(j(a))=\ldbrack \varphi(j(a), j(\bar{b}))\rdbrack_{\mathbb{L}_{\alpha}^{\mathbb{Q}}}=1$ and $j(Y)(j(a))=\ldbrack \psi(j(a), j(\bar{d}))\rdbrack_{\mathbb{L}_{\alpha}^{\mathbb{Q}}}=0$
    \end{center}
    
    i.e., $j(X)\not=j(Y)$ and $j\restriction_{\mathbb{L}_{\alpha}^{\mathbb{Q}}}$ is injective.\\
    
    Let us now show that $j\restriction_{\mathbb{L}_{\alpha}^{\mathbb{Q}}}$ is surjective in the sense of item 3. of this theorem. Let $Y\in \mathbb{L}_{\alpha+1}^{\mathbb{Q}}\setminus\mathbb{L}_{\alpha}^{\mathbb{Q}}=Def^{\mathbb{Q}}(\mathbb{L}_{\alpha}^{\mathbb{Q}})$. Then, there is a $\mathcal{L}_{\in}$-formula $\psi(x, \bar{y})$ and parameters $\bar{d}\in L_{\alpha}^{|\bar{x}|}$ such that for all $c\in\mathbb{L}_{\alpha}^{\mathbb{Q}}=dom(Y)$
    
    \begin{center} 
    $Y(c)=\ldbrack \psi(c, \bar{d})\rdbrack_{\mathbb{L}_{\alpha}^{\mathbb{Q}}}$.
    \end{center}
    
    By the induction hypothesis item $3.$, there exists $\bar{b}\in L_{\alpha}^{|\bar{y}|}$ such that $\mathbb{L}_{\alpha}^{\mathbb{Q}}\models j( \bar{b})=\bar{d}$. Let us define
    
    \begin{center}
        $X:=\{a\in L_{\alpha}: L_{\alpha}\models \psi(a, \bar{b})\}\in L_{\alpha+1}$.
    \end{center}
    
    We then have two cases:
    
    \begin{enumerate}
        \item Suppose that $X\in L_{\alpha}$, then, we have that there is an ordinal $\beta<\alpha$ such that $X\in L_{\beta+1}$. Thus, we have that $X\subseteq L_{\beta}$ and that there is a first-order $\mathcal{L}_{\in}$-formula $\varphi(w, \bar{z})$ and parameters $\bar {f}\in L_{\alpha}^{|\bar{z}|}$ such that
        
        \begin{center}
            $X=\{e\in L_{\beta}: L_{\beta}\models \varphi(e, \bar{f})\}=\{a\in L_{\alpha}: L_{\alpha}\models \psi(a, \bar{b})\}$
        \end{center}
        
        Thus, since $L_{\beta}\subseteq L_{\alpha}$, we have that for all $e\in L_{\beta}$,
        
        \begin{center}
            $L_{\beta}\models \varphi(e, \bar{f})$ if and only if $L_{\alpha}\models \psi(e, \bar{b})$
        \end{center}
        
        by the induction hypothesis item 4., we get
        
        \begin{center}
            $\ldbrack \varphi(j(e), j(\bar{f}))\rdbrack_{\mathbb{L}_{\beta}^{\mathbb{Q}}}=\ldbrack \psi(j(e), j(\bar{b}))\rdbrack_{\mathbb{L}_{\alpha}^{\mathbb{Q}}}=\ldbrack \psi(j(e), \bar{d})\rdbrack_{\mathbb{L}_{\alpha}^{\mathbb{Q}}}$ for all $e\in L_{\beta}$
        \end{center}
        
        By using the ``surjectivity'' of $j$, we can generalize this to all $g\in dom(j(X))=\mathbb{L}_{\beta}^{\mathbb{Q}}$, obtaining
        
        \begin{center}
            $j(X)(g)=\ldbrack \varphi(g, j(\bar{f}))\rdbrack_{\mathbb{L}_{\beta}^{\mathbb{Q}}}=\ldbrack \psi(g, \bar{d})\rdbrack_{\mathbb{L}_{\alpha}^{\mathbb{Q}}}=Y(g)$ for all $g\in dom(j(X))$
        \end{center}
        
        In this way, we conclude that $Y$ is an extension of $j(X)$. Let us see how $Y$ behaves in $\mathbb{L}_{\alpha}^{\mathbb{Q}}\setminus\mathbb{L}_{\beta}^{\mathbb{Q}}$. Given $a\in L_{\alpha}\setminus L_{\beta}$, since $a\not \in L_{\beta}$ and $X\subseteq L_{\beta}$, we have $a\not \in X$ and therefore 
        
        \begin{center}
            $L_{\alpha}\nvDash \psi(e, \bar{b})$, i.e., $\ldbrack \psi(j(e), j(\bar{b}))\rdbrack_{\mathbb{L}_{\alpha}^{\mathbb{Q}}}=\ldbrack \psi(j(e),\bar{d})\rdbrack_{\mathbb{L}_{\alpha}^{\mathbb{Q}}}=0.$
        \end{center}
        
        By the ``surjectivity'' of $j$, we can generalize the equality given above to
        
        \begin{center}
            $j(Y)=\ldbrack \psi(g,\bar{d})\rdbrack_{\mathbb{L}_{\alpha}^{\mathbb{Q}}}=0$ for all $g\in \mathbb{L}_{\alpha}^{\mathbb{Q}}\setminus\mathbb{L}_{\beta}^{\mathbb{Q}}$.
        \end{center}
        
        In this way, since $Y$ is an extension of $j(X)$ such that for all $g\in dom(Y)\setminus dom(j(X))$, $Y(g)=0$, we conclude that the functions $j(X)$ and $Y$ are equal in the sense of $\ldbrack \cdot=\cdot\rdbrack$ by Lemma \ref{lemaextensionigualdad}, i.e.,
        
        \begin{center}
            $\mathbb{L}_{\alpha}^{\mathbb{Q}}\models j(X)=Y$
        \end{center}
        as we wanted.

        \item Suppose $X\in L_{\alpha+1}\setminus L_{\alpha}$. In this case, we have $j(X)=Y$ since $dom(j(X))=\mathbb{L}_{\alpha}^{\mathbb{Q}}=dom(Y)$ and for all $c\in\mathbb{L}_{\alpha}^{\mathbb{Q}}$,
    
    \begin{center}
        $j(X)(c)=\ldbrack \psi(c, j(\bar{b}))\rdbrack_{\mathbb{L}_{\alpha}^{\mathbb{Q}}}=\ldbrack \psi(c, \bar{d})\rdbrack_{\mathbb{L}_{\alpha}^{\mathbb{Q}}}=Y(c)$, i.e., $j(X)=Y$
    \end{center}
    
    \end{enumerate}
    
    Let us see that we have the property $4$ by induction on formulas:
    
    \begin{enumerate}
        \item $\in$: Let $X, Y\in L_{\alpha+1}$, we see that
        \begin{center}
            $L_{\alpha+1}\models X\in Y$ if and only if $\mathbb{L}_{\alpha+1}^{\mathbb{Q}}\models j(X)\in j(Y)$.
        \end{center}
        
        $(\Rightarrow)$ Suppose $L_{\alpha+1}\models X\in Y$. Thus, since $X\in Y\subseteq L_{\alpha}$, we have $X\in L_{\alpha}$. Consider an $\mathcal{L}_{\in}$-formula $\psi(x, \bar{y})$ and parameters $\bar{b}\in L_{\alpha}^{|\bar{x}|}$ such that 
        
        \begin{center}
            $Y=\{a\in L_{\alpha}: L_{\alpha}\models \psi(a, \bar{b})\}$
        \end{center}
        
        Then, since $X\in Y$, we have $L_{\alpha}\models \psi(X, \bar{b})$ and therefore, by the induction hypothesis, $\mathbb{L}_{\alpha}^{\mathbb{Q}}\models \psi(j(X), j(\bar{b}))$, i.e.,
        
        \begin{center}
            $j(Y)(j(X))=\ldbrack \psi(j(X, j(\bar{b}))\rdbrack_{\mathbb{L}_{\alpha}^{\mathbb{Q}}}=1$
        \end{center}
        
        And since $\ldbrack j(X)=j(X)\rdbrack_{\mathbb{L}_{\alpha}^{\mathbb{Q}}}=1$, we have $j(Y)(j(X))\land \ldbrack j(X)=j(X)\rdbrack_{\mathbb{L}_{\alpha}^{\mathbb{Q}}}=1$ and since $j(X) \in dom(Y)\subseteq \mathbb{L}_{\alpha}^{\mathbb{Q}}$, we have
        
        \begin{center}
            $\ldbrack j(X)\in j(Y)\rdbrack_{\mathbb{L}_{\alpha}^{\mathbb{Q}}}=\bigvee\limits_{c\in dom(j(Y))}j(Y)(c)\land \ldbrack c=j(X)\rdbrack_{\mathbb{L}_{\alpha}^{\mathbb{Q}}}=1$
        \end{center}
        
        as desired.
        
        $(\Leftarrow)$ Let us suppose that
        
        \begin{center}
            $\ldbrack j(X)\in j(Y)\rdbrack_{\mathbb{L}_{\alpha}^{\mathbb{Q}}}=\bigvee\limits_{c\in dom(j(Y))}j(Y)(c)\land \ldbrack c=j(X)\rdbrack_{\mathbb{L}_{\alpha}^{\mathbb{Q}}}=1$
        \end{center}

       Thus, since $j(Y)(c)\land \ldbrack c=j(X)\rdbrack_{\mathbb{L}_{\alpha}^{\mathbb{Q}}}$ can only be $0$ or $1$, there exists $c\in dom(j(Y))\subseteq \mathbb{L}_{\alpha}^{\mathbb{Q}}$ such that
        
        \begin{center}
            $j(Y)(c)\land \ldbrack c=j(X)\rdbrack_{\mathbb{L}_{\alpha}^{\mathbb{Q}}}=1$,
        \end{center}
        then, 
        \begin{center}
            $j(Y)(c)=1$ and $\ldbrack c=j(X)\rdbrack_{\mathbb{L}_{\alpha}^{\mathbb{Q}}}=1$
        \end{center}
    
        Since $j(Y)(c)=\ldbrack \psi(c, j(\bar{b}))\rdbrack_{\mathbb{L}_{\alpha}^{\mathbb{Q}}}=1$ and $\ldbrack c=j (X)\rdbrack_{\mathbb{L}_{\alpha}^{\mathbb{Q}}}=1$, we have that
        
        \begin{center}
            $j(Y)(j(X))=\ldbrack \psi(j(X), j(\bar{b}))\rdbrack_{\mathbb{L}_{\alpha}^{\mathbb{Q}}}=1$.
        \end{center}
        By induction hypothesis, we have that
        
        \begin{center}
            $L_{\alpha}\models\psi(X, \bar{b})$
        \end{center}
        
        and then $X\in Y$.

        \item Equality: Let $X, Y\in L_{\alpha+1}\setminus L_{\alpha}$. Let us see that
         \begin{center}
             $L_{\alpha+1}\models X= Y$ if and only if $\mathbb{L}_{\alpha+1}^{\mathbb{Q}}\models j(X)=j(Y)$.
         \end{center}

         By definition, we have that there exist first-order $\mathcal{L}_{\in}$-formulas $\varphi(x, \bar{y})$, $\psi(z, \bar{w})$ and parameters $\bar{ b}\in L_{\alpha}^{|\bar{x}|}$ and $\bar{d}\in L_{\alpha}^{|\bar{w}|}$ such that

    \begin{center}
        $X=\{a\in L_{\alpha}: L_{\alpha}\models \varphi(a, \bar{b})\}$ and $Y=\{a\in L_{\alpha}: L_{\alpha}\models \psi(a, \bar{d})\}$.
    \end{center}
        
        If $X=Y$, then, since $j$ is a function, $j(X)=j(Y)$ and therefore $\mathbb{L}_{\alpha+1}^{\mathbb{Q}}\models j(X)= j(Y)$.
        
         Suppose now that $X\not=Y$. Without loss of generality, suppose there exists $a\in X$ with $a\not\in Y$, i.e.,
        
        \begin{center}
            $L_{\alpha}\models \varphi(a, \bar{b})$ and $L_{\alpha}\nvDash \psi(a, \bar{d})$
        \end{center}
       By the induction hypothesis item 4., this means that
        
        \begin{center}
            $j(X)(j(a))=\ldbrack \varphi(j(a), j(\bar{b}))\rdbrack_{\mathbb{L}_{\alpha}^{\mathbb{Q}}}=1$ and $j(Y)(j(a))=\ldbrack \psi(j(a), j(\bar{d}))\rdbrack_{\mathbb{L}_{\alpha}^{\mathbb{Q}}}\not=1$.
        \end{center}
        but since $\ldbrack \psi(j(a), j(\bar{d}))\rdbrack_{\mathbb{L}_{\alpha}^{\mathbb{Q}}}$ can only take values in $\{0, 1\}$, we get
        \begin{center}
            $j(Y)(j(a))=\ldbrack \psi(j(a), j(\bar{d}))\rdbrack_{\mathbb{L}_{\alpha}^{\mathbb{Q}}}=0.$
        \end{center}
        
        Let us see that $\ldbrack j(a)\in j(Y)\rdbrack_{\mathbb{L}_{\alpha}^{\mathbb{Q}}}=0$.
        
        Notice that
        
        \begin{center}
            $j(X)(j(a))=1$ and $\ldbrack j(a)\in j(Y)\rdbrack_{\mathbb{L}_{\alpha}^{\mathbb{Q}}}=0$, i.e., $j(X)(j(a))\rightarrow\ldbrack j(a)\in j(Y)\rdbrack_{\mathbb{L}_{\alpha}^{\mathbb{Q}}}=0$
        \end{center}
        
        and therefore
        
        \begin{center}
            $\ldbrack j(X)\subseteq j(Y)\rdbrack_{\mathbb{L}_{\alpha}^{\mathbb{Q}}}=\bigwedge\limits_{c\in dom(j(X))}j(X)(c)\rightarrow\ldbrack c\in j(Y)\rdbrack_{\mathbb{L}_{\alpha}^{\mathbb{Q}}}=0$
        \end{center}
        and we have $\ldbrack j(X)=j(Y)\rdbrack_{\mathbb{L}_{\alpha}^{\mathbb{Q}}}=0$, as desired.
    \end{enumerate}
    The rest of the induction is straightforward, and therefore, by induction on formulas, we have the theorem.
    
    \end{proof}

\bibliographystyle{amsalpha}

\begin{thebibliography}{99}

\bibitem{Bell2005}Bell, J. Set Theory: Boolean-Valued Models and Independence Proofs. (Oxford University Press UK, 2005)
\bibitem{BuPi2014}Buşneag, D. \& Piciu, D. Some types of filters in residuated lattices. {\em Soft Computing - A Fusion Of Foundations, Methodologies And Applications}. \textbf{18}, 825-837 (2014)
\bibitem{CarnielliConiglio2019}Carnielli, W. \& Coniglio, M. Twist-Valued Models for Three-valued Paraconsistent Set Theory.  (2019)
\bibitem{Chang1971}Chang, C. Sets constructible using $\mathcal{L}_{\kappa, \kappa}$. {\em Axiomatic Set Theory (Proc. Sympos. Pure Math.}. \textbf{Vol. XIII} pp. 1-8 (1971)
\bibitem{Dilworth1939}Dilworth, R. Non-Commutative Residuated Lattices. {\em Transactions Of The American Mathematical Society}. \textbf{46}, 426-444 (1939)
\bibitem{Fitting1969}Fitting, M. Intuitionistic Logic, Model Theory and Forcing. (North-Holland Pub. Co, 1969)
\bibitem{Godel1938}Gödel, K. The Consistency of the Axiom of Choice and of the Generalized Continuum-Hypothesis. {\em Proceedings Of The National Academy Of Sciences Of The United States Of America}. \textbf{24}, 556-557 (1938)
\bibitem{Grayson1979}Grayson, R. Heyting-valued models for intuitionistic set theory. {\em Applications Of Sheaves: Proceedings Of The Research Symposium On Applications Of Sheaf Theory To Logic, Algebra, And Analysis}. pp. 402-414 (1979)
\bibitem{HajekHanikova2001}Hájek, P. \& Haniková, Z. A set theory within fuzzy logic. {\em Proceedings Of The International Symposium On Multiple-Valued Logic}. pp. 319-323 (2001)
\bibitem{HajekHanikova2003}Hájek, P. \& Haniková, Z. A Development of Set Theory in Fuzzy Logic. {\em Beyond Two: Theory And Applications Of Multiple-Valued Logic}. \textbf{114} pp. 273-285 (2003)
\bibitem{Hohle1994}Höhle, U. Monoidal Logic. {\em Fuzzy-Systems In Computer Science}. pp. 233-243 (1994)
\bibitem{VenturiMartinez2021}Jockwich Martinez, S. \& Venturi, G. Non-classical Models of ZF. {\em Studia Logica}. \textbf{109}, 509-537 (2021)
\bibitem{KMV2016}Kennedy, J., Magidor, M. \& Väänänen, J. Inner Models from Extended Logics: Part 1. {\em  Avalaible At Https://arxiv.org/pdf/2007.10764.pdf}. (2020)
\bibitem{Lano1992a}Lano, K. Fuzzy sets and residuated logic. {\em Fuzzy Sets And Systems}. \textbf{47}, 203 - 220 (1992)
\bibitem{MacCaull1996}MacCaull, W. A note on Kripke semantics for residuated logic. {\em Fuzzy Sets And Systems}. \textbf{77}, 229-234 (1996)
\bibitem{Moncayo2023}Moncayo, J. Constructible sets in lattice-valued models.. {\em (Master Dissertation) Universidad Nacional De Colombia, Bogotá}. (2023)
\bibitem{Mulvey1986}Mulvey, C. \&. {\em Rendiconti Del Circolo Matematico Di Palermo}. \textbf{12}, 99-104 (1986)
\bibitem{MoncayoZambrano2X}Moncayo, J. \& Zambrano, P. Kripke-like models of Set Theory in Modal Residuated Logic. , in process.
\bibitem{Scott1971}Scott, D. \& Myhill, J. Ordinal Definability. {\em Axiomatic Set Theory: Proceedings Of Symposia In Pure Mathematics}. \textbf{Vol. XIII} pp. 271-278 (1971)
\bibitem{SS1967}Scott, D. \& Solovay, R. Boolean-Valued Models for Set Theory. {\em Mimeographed Notes For The 1967 American Math. Soc. Symposium On Axiomatic Set Theory}. (1967)
\bibitem{Ward1938}Ward, M. Structure Residuation. {\em Annals Of Mathematics}. \textbf{39}, 558-568 (1938)
\bibitem{WardDilworth1938}Ward, M. \& Dilworth, R. Residuated Lattices. {\em Proceedings Of The National Academy Of Sciences Of The United States Of America}. \textbf{24}, 162-164 (1938)


\end{thebibliography}

\end{document}